\documentclass[12pt]{article}
\usepackage[utf8]{inputenc}
\usepackage[legalpaper, margin=1in]{geometry}
\usepackage[pdftex]{pict2e}

\usepackage{times}
\usepackage{epsfig}
\usepackage{graphicx}
\usepackage{amsmath}
\usepackage{amssymb}
\usepackage{tabularx} 
\usepackage{mathtools} 
\usepackage{amsfonts}
\usepackage{amsthm}
\usepackage{booktabs}
\usepackage{authblk}
\usepackage{tikz-cd}
\usepackage[ruled,vlined,linesnumbered]{algorithm2e}

\newtheorem{innercustomgeneric}{\customgenericname}
\providecommand{\customgenericname}{}
\newcommand{\newcustomtheorem}[2]{%
  \newenvironment{#1}[1]
  {%
   \renewcommand\customgenericname{#2}%
   \renewcommand\theinnercustomgeneric{##1}%
   \innercustomgeneric
  }
  {\endinnercustomgeneric}
}

\newcustomtheorem{customtheorem}{Theorem}
\newcustomtheorem{customproposition}{Proposition}
\newcustomtheorem{customlemma}{Lemma}
\newcustomtheorem{customcorollary}{Corollary}

\usepackage[draft]{fixme}
\fxsetup{theme=color, mode=multiuser, noinline}
\FXRegisterAuthor{ec}{EC}{\color{purple}Erin}
\FXRegisterAuthor{fr}{FR}{\color{blue}Felix}

\usepackage[pagebackref=true,breaklinks=true,letterpaper=true,colorlinks,bookmarks=false]{hyperref}

\usepackage{tabularx} 
\usepackage{amsmath}
\usepackage{graphicx} 
\usepackage{amsmath}
\usepackage{mathtools} 
\usepackage{amsfonts}
\usepackage{amssymb}
\usepackage{amsthm}
\usepackage{enumitem}
\usepackage{dsfont}
\usepackage{tikz-cd} 
\usepackage{comment}
\usepackage[utf8]{inputenc}
\usepackage{changepage}
\usepackage[format=plain, font=footnotesize]{caption}
\usepackage[sort]{cite}
\usepackage{ulem}

\hypersetup{
	colorlinks=true,       
	linkcolor=teal,        
	citecolor=orange,        
	filecolor=magenta,     
	urlcolor=blue         
}
\usepackage{blindtext}

\newcommand{\PP}{\mathbb{P}}
\newcommand{\RR}{\mathbb{R}}

\newcommand{\Set}[1]{\left\{#1\right\}}

\usepackage[capitalise, noabbrev, nameinlink]{cleveref}

\theoremstyle{plain}

\newtheorem{theorem}{Theorem}[section] 
\newtheorem{proposition}[theorem]{Proposition}
\newtheorem{lemma}[theorem]{Lemma}

\newtheorem{corollary}[theorem]{Corollary}

\newtheorem{definition}[theorem]{Definition} 

\newtheorem{remark}[theorem]{Remark}

\begin{document}
\title{Geometric Interpretations of Compatibility for Fundamental Matrices}
\author{Erin Connelly\quad Felix Rydell}

\maketitle

\begin{abstract} 
In recent work, algebraic computational software was used to provide the exact algebraic conditions under which a sixtuple $\{F^{ij}\}$ of fundamental matrices, corresponding to $4$ images, will be compatible, i.e. there will exist cameras $\{P_i\}_{i=1}^4$ such that each pair $P_i,P_j$ has fundamental matrix $F^{ij}$; it has been further demonstrated that quadruplewise compatibility is sufficient for the problem of $n>4$ images. We expand on these prior results by proving equivalent geometric conditions for compatibility. We find that when the camera centers are in general position, compatibility can be characterized via the intersections of epipolar lines in each image. When the camera centers are coplanar, compatibility occurs when the prior condition holds and additionally any one camera center can be reconstructed via the other three.
\end{abstract}



\section{Introduction} 

The \textit{structure-from-motion} pipeline in computer vision aims to build 3D models based on 2D images. This pipepline generally falls into two categories, incremental \cite{snavely2006photo} and global \cite{cui2015global}. In either case, fundamental matrices are central.  

Before we define what a fundamental matrix, we need some notation. Define \textit{projective space} $\PP^n=\PP(\RR^{n+1})$ as the \textit{projectivization} of $\RR^{n+1}$, i.e. the quotient space $(\RR^{n+1}\setminus\{0\})/\sim$ with respect to the equivalence relation that $X\sim Y$ if and only if $X$ and $Y$ are parallel. We use $\sim$ to denote equality in $\PP^n$. We refer to full rank $3\times 4$ projective matrices as \textit{cameras}. The affine kernel of a camera $P$ is spanned by one non-zero vector in $\RR^4$, and its projective class $\ker P$ is its \textit{center}. We denote by $\mathrm{GL}_n$ the set of invertible $n\times n$ matrices and by $\mathrm{PGL}_n$ its projectivization.

Consider the rational map $\psi : \PP(\RR^{3\times 4})\times \PP(\RR^{3\times 4})\dashrightarrow \PP(\RR^{3\times 3})$ defined as follows. Given a pair of $3\times 4$ projective matrices ${P}_1$ and ${P}_2$, the determinant
\begin{align}\label{eq: det}
\det\begin{bmatrix}
{P}_1&\textbf{x}&0\\
{P}_2&0&\textbf{y}
\end{bmatrix}
\end{align}
is a bilinear polynomial in $\textbf{x}$ and $\textbf{y}$, meaning there is a projective matrix ${F}^{12}$ such that \Cref{eq: det} can be written as $\textbf{x}^\top {F}^{12}\textbf{y}$. We define $\psi({P}_1,{P}_2)$ to be this $3\times 3$ projective matrix. If $P_1,P_2$ are cameras with distinct centers, we say that $F^{12}=\psi(P_1,P_2)$ is the \textit{fundamental matrix} of $P_1,P_2$. This map is undefined, i.e. \Cref{eq: det} is the 0 polynomial, precisely when the affine kernels of $P_1$ and $P_2$ intersect outside of $0$. 

Given two cameras $P_1,P_2$ with distinct centers $c_1,c_2$, for any point $p\in\PP^2$ the composition $P_1(P_2^{-1}(p))$ will yield a line in the first image. In particular, this line will pass through the \textit{epipole} $e_1^2:=P_1c_2$. This correspondence from points in one image to lines in the other is given explicitly by the fundamental matrix, in the sense that $F^{12}x$ defines the line $\{y\in\PP^2~:~y^\top F^{12}x=0\}$.

\begin{proposition}[{\hspace{1sp}\cite[Section 9]{Hartley2004}}]\label{prop:properties_of_the_fundamental_matrix} $ $

\begin{enumerate} 
    \item For any rank $2$ matrix $F^{12}\in\PP(\RR^{3\times 3})$, there exists two cameras $P_1,P_2$ such that $F^{12}$ is their fundamental matrix. All other cameras $C_1,C_2$ with fundamental matrix $F^{12}$ satisfy $C_1\sim P_1H,C_2\sim P_2H$ for some $H\in \mathrm{PGL}_4$;
    \item $\psi({P}_2,{P}_1)\sim \psi({P}_1,{P}_2)^\top $;
    \item If $F^{12}$ is the fundamental matrix of $P_1,P_2$, then $\ker F^{12}\sim P_2 \ker P_1$\label{enum: epip};
    \item For cameras $P_1,P_2$, we have $F^{12}\sim\psi(P_1,P_2)$ if and only if $P_1^\top F^{12}P_2$ is a skew-symmetric matrix.
\end{enumerate}
\end{proposition}
In absence of the original cameras, we define the $i$-th \textit{epipole} $e_j^i$ in the $j$-th image to be $\ker F^{ij}$, which is a single point in $\PP^2$. By \Cref{enum: epip}, $e_j^i$ is the image of the $i$-th camera center taken by the $j$-th camera, so this is consistent with our previous definition above. Although fundamental matrices and epipoles are only defined up to scale, i.e. as elements in projective space, when working with them, we usually assume that we are given affine representatives such that $(F^{ij})^\top =F^{ji}$, unless otherwise specified.

\begin{definition}
A set of fundamental matrices $\Set{F^{ij}}$ is \textnormal{compatible} if there exist cameras $P_1,\ldots,P_n$ such that $F^{ij}=\psi(P_i,P_j)$. The set of cameras $\Set{P_i}$ is called a \textnormal{solution} to $\Set{F^{ij}}$.
\end{definition}

For $n=2$, given any fundamental matrix there is a solution of cameras $P_1$ and $P_2$, which is unique up to global projective transformation. For the case $n=3$ with three fundamental matrices $F^{12},F^{13}$ and $F^{23}$, there generically is no solution. Indeed, a classical result \cite[Section 15.4]{Hartley2004} provides triplewise constraints on $F^{12},F^{13},F^{23}$ in terms of the fundamental matrices and their epipoles. In general, for a set of ${n\choose 2}$ fundamental matrices $\Set{F^{ij}}$, conditions have been given in terms of $n$-view matrices in \cite[Theorem 1]{kasten2019gpsfm} and \cite[Theorem 2]{geifman2020averaging}. More recently, in \cite{braatelund2023compatibility}, it was proven that quadruplewise compatibility implies global compatibility, and explicit homogeneous polynomials were provided in the case of $n=4$ via a computational algebra package. 

The main contribution of this paper is a geometric characterization of the compatibility conditions from \cite{braatelund2023compatibility}, and also the simplification of several arguments. In \cref{thm: 4tuple-condition} we show that compatibility in the case of $4$ images with cameras in general position is characterized via three epipolar lines intersecting in a single point. We first prove this statement geometrically and then re-derive the corresponding algebraic result from \cite{braatelund2023compatibility}. In \cref{thm: Case 2 general form} we show that compatibility in the case of $4$ images with coplanar cameras is characterized by two different geometric conditions, each concerning the mutual intersections of three lines; one of these conditions is the epipolar line condition from the above case and the other condition concerns lines in the world space. We again first prove this statement geometrically before re-deriving the corresponding algebraic result from \cite{braatelund2023compatibility}.

In \Cref{s: Pre}, we recall notation and previous work. In \Cref{s: Geo int 3}, we give geometric interpretations for triplewise conditions. In \Cref{s: Geo int 4} we give geometric interpretations for quadruplewise conditions.

\subsection*{Related work} Camera matrices are often assumed to be \textit{calibrated}, represented as $[R|t]$ for a rotation matrix $R$ and a translation vector $t$. The corresponding fundamental matrices are called \textit{essential matrices}. Compatibility for essential matrices has been shown by \cite{kasten2019algebraic} to provide a necessary and sufficient condition for compatibility of essential matrices, in terms of the $n$-view essential matrix. In \cite{martyushev2020necessary}, Martyushev provides explicit homogeneous polynomials that describe necessary and sufficient condition for compatibility of three essential matrices.

We finally note the relation to the question of \textit{solvability}. A viewing graph is considered \textit{solvable} if, given a generic set of cameras, their fundamental matrices have a unique solution in terms of cameras up to global projective transformation. Recent work on this topic include \cite{trager2018solvability,arrigoni2021viewing}. Furthermore, solvability has been investigated in the case of calibrated cameras, where it is known that the solvable graphs are precisely those that are parallel rigid \cite{ozyesil2015robust,sattler2016efficient}.

From the practical point of view, \cite{kasten2019gpsfm} proposed an algorithm for projective structure-from-motion that employs their necessary and sufficient condition for compatibility. The question of compatibility also arises in the study of critical configurations \cite{braatelund2021critical,HK}.

\paragraph{Acknowledgements.} Felix Rydell was supported by the Knut and Alice Wallenberg Foundation within their WASP (Wallenberg AI, Autonomous Systems and Software Program) AI/Math initiative. 

 \section{Preliminaries}\label{s: Pre} \numberwithin{equation}{section} In this section we establish notation used throughout our paper.
 Recall that given a vector $t\in \RR^3$, we define 
\begin{align}
    [t]_\times=\begin{bmatrix} 0 & -t_3& t_2\\
t_3 & 0 & -t_1\\ -t_2 &t_1 &0\end{bmatrix}.
\end{align}
With respect to the cross product $\times$ on $\RR^3\times \RR^3$, we have $t\times u=[t]_\times u$.

\begin{lemma}[{\hspace{1sp}\cite[Results 9.15]{Hartley2004}}]\label{le: can sol} Let $F^{12}$ be a fundamental matrix. Then
\begin{align}
    P_1=\begin{bmatrix}[e_1^2]_\times F^{12}+e_1^2v^\top & \lambda e_1^2\end{bmatrix}, \quad P_2=\begin{bmatrix}I&0\end{bmatrix}
\end{align}
is a solution of cameras for any $v$ and non-zero $\lambda$. Moreover this is an exhaustive list of solutions in $P_1$ for fixed $P_2=\begin{bmatrix}
    I & 0
\end{bmatrix}$.     
\end{lemma}



An important tool for the study of compatibility is coordinate changes in both the world and image coordinates. This fact has been used for instance in \cite{martyushev2020necessary,braatelund2023compatibility}. To be precise, $\mathrm{PGL}_3^n$ (or equivalently $\mathrm{GL}_3^n$) acts on a set of fundamental matrices $\Set{F^{ij}}$ by
\begin{align}
    \Set{F^{ij}}\mapsto \Set{H_i^\top F^{ij}H_j}.
\end{align}
We call this the \textit{fundamental action} of $\mathrm{PGL}_3^n$. The main appeal of this action is that we can use it to simplify a set of fundamental matrices, without affecting compatibility:

\begin{proposition}[{\hspace{1sp}\cite[Proposition 2.1]{braatelund2023compatibility}}] 
\label{prop:fundamental_action_preserves_compibility} Let $\Set{F^{ij}}$ be a set of fundamental matrices. Let $P_i$ be a solution to $\Set{F^{ij}}$. For any $(H_1,\ldots,H_n,H)\in \mathrm{PGL}_3^n\times \mathrm{PGL}_4$, we have, 
\begin{align}\label{eq: fund}
\psi(H_i^{-1}P_iH,H_j^{-1}P_jH)\sim H_i^\top \psi(P_i,P_j)H_j.
\end{align} 
In particular, $\Set{F^{ij}}$ is compatible if and only if $\Set{G^{ij}}$ is compatible, where $G^{ij}:=H_i^\top F^{ij}H_j$. 
\end{proposition}

We refer to quantities of the form $\textbf{e}_{sijt}:=(e_i^s)^\top F^{ij}e_j^t$ as \textit{epipolar numbers}. These epipolar numbers will feature in many of our equations and can easily be checked to be invariant under the fundamental action. Next, we recall that it suffices to characterize compatibility for triplets and six-tuples of fundamental matrices.

\begin{theorem}[{\hspace{1sp}\cite[Theorem 3.15]{braatelund2023compatibility}}]
\label{thm:compatible_if_each_six-tuple_is_compatible}
Let $\Set{F^{ij}}$ be a complete set of $n\choose2$, $n\ge 4$, fundamental matrices such that 
for all $i,j,k,l$, the six-tuple $F^{ij},F^{ik},F^{jk},F^{il},F^{jl},F^{kl}$ is compatible. Then $\Set{F^{ij}}$ is compatible.

Moreover, if all epipoles in each image coincide, then triplewise compatibility implies that $\Set{F^{ij}}$ is compatible. The reconstruction in this case will be a set of cameras whose centers all lie on a line.
\end{theorem}

Finally, we define the following terminology: The \textit{back-projected line} of an image point $x\in \PP^2$ with respect to a camera $C$ with center $c$ is the line $C^{-1}(x)\cup\{p\}$. It is the span of $C^\dagger x$ and $p$, for any pseudo-inverse $C^\dagger$ such that $CC^\dagger=I$. For $\PP^n$, a projective \textit{frame} or \textit{basis} is a set of $n+2$ points such that no $n+1$ are contained in a hyperplane. A canonical choice is $p_1=(1,0,\ldots,0),p_{n+1}=(0,\ldots,0,1)$, and $p_{n+2}=(1,\ldots,1)$. It is a basic fact of projective geometry that given two projective frames $\{a_1,\ldots, a_{n+2}\}$ and $\{b_1,\ldots,b_{m+2}\}$ for $\PP^n$ and $\PP^m$ respectively, with $n\geq m$, there is a unique projective transformation $M$ such that $Ma_i=b_i$ for all $i=1\ldots,m$ and $Ma_i=0$ for all $i>m$.



\section{Geometric Interpretations for Triplewise Constraints}\label{s: Geo int 3} 

In this section we give geometric interpretations for triplewise conditions. Given a set of three cameras with pairwise distinct camera centers, we can characterize their geometry as fitting into one of two cases; either they are non-collinear or they are collinear. This geometry is also captured by the epipoles in each image. Indeed, by the fact that $e_i^j\sim P_i\ker P_j$, we can deduce what possible configuration of epipoles are may occur for compatible fundamental matrices. We formally define the cases as follows: 

\begin{enumerate}[leftmargin =8.8em]
    \item [Non-Collinear Case:] The cameras are in generic position, meaning no line contains all centers. Equivalently, in each image, the two epipoles are distinct.
    \item [Collinear Case:] All camera centers lie in a line. Equivalently, in each image, the two epipoles are coincident.
\end{enumerate}

\begin{theorem}[Non-Collinear Case]
\label{thm: non colin}
Let $F^{12}$, $F^{13}$, $F^{23}$ be fundamental matrices such that the two epipoles in each image are distinct. Then $\Set{F^{ij}}$ is compatible if and only if
\begin{align}
\label{eq_non_collinear}
(e_{1}^{3})^{T}F^{12}e_{2}^{3}=(e_{1}^{2})^{T}F^{13}e_{3}^{2}=(e_{2}^{1})^{T}F^{23}e_{3}^{1}=0.
\end{align}

Further, $(e_i^k)^\top F^{ij}e_j^k=0$ is equivalent to the condition that after reconstructing two of the cameras $C_i,C_j$, the third camera center $p_k$ can be recovered, i.e. the back-projected lines $C_i^{-1}e_i^k,C_j^{-1}e_j^k$ meet in a unique point away from the line spanned by the centers $p_i,p_j$ of $C_i,C_j$. 
\end{theorem}

The fact that \Cref{eq_non_collinear} characterizes compatibility in the non-collinear case is well-known \cite[Section 15.4]{Hartley2004}. Our contribution is the last part of the statement, although we include a proof for the first statement as well for the sake of completeness. In our proofs, typically work with affine representatives of fundamental matrices, cameras, centers and epipoles, and we write $=$ to denote affine equality and $\sim$ to denote projective equality (equality up to non-zero scaling).

\begin{proof} We begin by showing the equivalence of the conditions for $(i,j)=(1,2)$. A standard reconstruction of $C_1,C_2$ is given by $C_1=\begin{bmatrix} [e_1^2]_\times F^{12}&e_1^2   
\end{bmatrix}$ and $C_2=\begin{bmatrix}I &0 \end{bmatrix}$ by \Cref{le: can sol}. Note that $[e_1^2]_\times F^{12}$ is rank 2 matrix and its kernel is $e_2^1$. The back-projected line $C_2^{-1}e_2^3$ can be parametrized as $\mu_0p_2+\mu_1 C_2^\dagger e_2^3$, where $p_2$ is the center of $C_2$ and $C_2^\dagger$ is any matrix such that $C_2C_2^\dagger=I$. In this case, we may put $C_2^\dagger=\begin{bmatrix}I &0 \end{bmatrix}^\top $, so that we get a parametrization
\begin{align}
    \mu_0(0,0,0,1)^\top +\mu_1 ((e_2^3)^\top ,0)^\top  \quad \textnormal{ for } \quad \mu=(\mu_0,\mu_1)\neq (0,0).
\end{align}
The back-projected line $C_1^{-1}e_1^3$ meets $C_2^{-1}e_2^3$ precisely when there is a solution in $\mu $ to 
\begin{align}
    C_1\big(\mu_0(0,0,0,1)^\top +\mu_1 ((e_2^3)^\top ,0)^\top \big)= e_1^3,
\end{align}
which is can be written $\mu_0e_1^2+\mu_1 [e_1^2]_\times  F^{12}e_2^3 e_1^3$. Since $e_1^2$ and $[e_1^2]_\times  F^{12}e_2^3$ are non-zero and linearly independent, that there exists a unique solution $(\mu_0,\mu_1)$ (with $\mu_1\neq 0$) is equivalent to the following determinantal expression:
\begin{equation}
\begin{vmatrix}
    \vert & \vert & \vert\\
    e_1^3& e_1^2 & [e_1^2]_\times F^{12}e_2^3 \\
    \vert & \vert & \vert
\end{vmatrix}=0
\end{equation}
However, the determinant equals $(e_1^3)^\top (e_1^2\times [e_1^2]_\times F^{12}e_2^3)$, which further equals $([e_1^2]_\times^2 e_1^3)^\top F^{12}e_2^3=0$. We note that $[e_1^2]_\times^2 e_1^3$ lies in the span of $e_1^2$ and $e_1^3$, showing that this identity is equivalent to $(e_1^3)^\top F^{12}e_2^3=0$. Further, the intersection point $X\in \PP^3$ of the lines $C_1^{-1}e_1^3$ and $C_2^{-1}e_2^3$ does not line in the span of $p_1$ and $p_2$. To see this, observe that $X$ satisfies $X\not\sim p_2$ (since $\mu_1\neq 0$ above) and $e_2^3=C_2X$. Because of $e_2^1\not \sim  e_2^3$, we deduce that $X$ does not lie in $C_2^{-1}e_2^1$; the line spanned by $p_1$ and $p_2$. 

$\Rightarrow)$ If a reconstruction of cameras exists, then each center can be recovered. More precisely, the intersection of $C_i^{-1}e_i^3$ and $C_j^{-1}e_j^3$ contains the center $p_3$. If $C_i^{-1}e_i^3=C_j^{-1}e_j^3$, then this would imply that the centers are collinear; a contradiction.

$\Leftarrow)$ Let $C_1,C_2$ be any solution to $F^{12}$ with centers $p_1,p_2$. By \Cref{eq_non_collinear} and the first part of the proof, the last center $p_3$ can be recovered. Let $X,Y$ be any points in $\PP^3$ that make a projective frame together with the centers $p_1,p_2$ and $p_3$. Define $x_i:=C_iX$ and $y_i:=C_iY$. It follows that $\{e_i^j,e_i^3,x_i,y_i\}$ is a projective frame for $\PP^2$, where $i=1,2$, and $i$ and $j$ are distinct. Next we argue that there are unique $x_3,y_3\in \PP^2$ such that 
\begin{align}
    x_1^\top F^{13}x_3=x_2^\top F^{23}x_3=0\quad  \textnormal{ and } \quad y_1^\top F^{13}y_3=y_2^\top F^{23}y_3=0.
\end{align}
We prove this for $x_3$. If there we no such unique $x_3$, then we would have $x_1^\top F^{13}\sim x_2^\top F^{23}$ and as a direct consequence $x_1^\top F^{13}e_3^2=0$. By \Cref{eq_non_collinear}, we have $F^{13}e_3^2\sim e_1^2\times e_1^3$, and we now get a contradiction: $x_1^\top (e_1^2\times e_1^3)=0$, which implies that $x_1$ lies in the span of $e_1^2$ and $e_1^3$. 

Next, we prove that $\{e_3^1,e_3^2,x_3,y_3\}$ is a projective frame of $\PP^2$. If $e_3^1,e_3^2$ and $x_3$ were collinear, then $x_i^\top F^{i3}x_3=0$ implies that $x_i^\top F^{i3}e_3^j=0$ for some distinct $i$ and $j$. As a consequence, this $x_i$ is in the span of $e_i^j$ and $e_i^3$, which is a contradiction. In particular, $e_3^j\neq x_3,y_3$. If say $e_3^1,x_3,y_3$ were collinear, then $F^{13}x_3\sim F^{13}y_3$. This however implies $x_1,y_1,e_1^3$ are collinear, which again is a contradiction, and $\{e_3^1,e_3^2,x_3,y_3\}$ must be a projective frame. 

There is now a unique full-rank $\widetilde{C}_3$ mapping $p_i$ to $e_3^i$, $X$ to $x_3$ and $Y$ to $y_3$. We denote the fundamental matrices of $C_1,C_2,\widetilde{C}_3$ by $F^{12},\widetilde{F}^{13},\widetilde{F}^{23}$, and prove that $\widetilde{F}^{ij}\sim F^{ij}$. We do the proof for $(i,j)=(1,3)$. Observe that
\begin{align}\begin{aligned}\label{eq: F tilde}
0=&\widetilde{F}^{13}e_3^1=F^{13}e_3^1,\quad & &e_1^3\times e_3^2\sim \widetilde{F}^{13}e_3^2\sim F^{13}e_3^2,\\
e_1^3\times x_1\sim&\widetilde{F}^{13}x_3\sim F^{13}x_3,\quad & &e_1^3\times y_1\sim\widetilde{F}^{13}y_3\sim F^{13}y_3,
\end{aligned}
\end{align}
and since $\{e_3^1,e_3^2,x_3,y_3\}$ is a projective frame, these identities uniquely define a projective transormation; $\widetilde{F}^{13}\sim F^{13}$, and we are done. 
\end{proof}

The next result is a geometric interpretation of \cite[Proposition 3.4]{braatelund2023compatibility} for the Collinear Case. 

\begin{theorem}[Collinear Case]\label{thm: K3 colin geo} Let $F^{12}$, $F^{13}$, $F^{23}$ be fundamental matrices such that the epipoles in each image are equal. Then $\Set{F^{ij}}$ is compatible if and only if (up to scaling)
\begin{align}\label{eq: collin epip}
F^{32}=F^{31}[e_1^2]_\times F^{12}.
\end{align}  

Further, \Cref{eq: collin epip} is equivalent to that for any reconstruction of $C_1,C_2$ with respect to $F^{12}$, and any point $X\in \PP^3$, the two epipolar lines $F^{31}C_1X$ and $F^{32}C_2X$ coincide. 
\end{theorem}

\begin{proof} We begin by showing the equivalence of the conditions. A standard reconstruction of $C_1,C_2$ is given by $C_1=\begin{bmatrix} [e_1^2]_\times F^{12}&e_1^2   
\end{bmatrix}$ and $C_2=\begin{bmatrix}I &0 \end{bmatrix}$ by \Cref{le: can sol}. That $F^{32}C_2X$ and $F^{31}C_1X$ are equal for each $X\in \PP^3$ is equivalent to $F^{32}C_2=F^{31}C_1$. However, 
\begin{align}
    F^{32}C_2=\begin{bmatrix}
        F^{32} & 0
    \end{bmatrix}, \quad   F^{31}C_1=\begin{bmatrix}
        F^{31}[e_1^2]_\times F^{12} & 0
    \end{bmatrix},
\end{align}
since $e_1^2=e_1^3$. Then up to scaling, we get $F^{32}=F^{31}[e_1^2]_\times F^{12}$. 

$\Rightarrow)$ We have $(e_3^i)^\top F^{3i}C_iX=0$ and $(C_3X)^\top F^{3i}C_iX=0$ for each $i$ and $X$. As a consequence, for $X\in \PP^3$ away from the the line spanned by the centers, $F^{3i}C_iX\sim e_3^i\times (C_3X)\neq 0$. Since $e_3^1\sim e_3^2$, it follows that $F^{31}C_1X\sim F^{32}C_2X$, and this equality must hold for all $X$.

$\Leftarrow)$ By \Cref{le: can sol}, setting $C_3=\begin{bmatrix} [e_3^2]_\times F^{32}+e_3^2v^\top &\lambda e_3^2 
\end{bmatrix}$ with $C_1,C_2$ as above, we are left to show that $C_1,C_3$ is a solution to $F^{13}$. By assumption of $F^{32}=F^{31}[e_1^2]_\times F^{12}$ (and therefore  $F^{23}=-F^{21}[e_1^2]_\times F^{13}$),
\begin{align}
   C_1^\top F^{13}C_3=\begin{bmatrix}
        -F^{21}[e_1^2]_\times \\ (e_1^2)^\top 
    \end{bmatrix}F^{13}\begin{bmatrix}
        [e_3^2]_\times F^{32} +e_3^2v^\top & \lambda e_3^2
    \end{bmatrix} =\begin{bmatrix}
        F^{23}[e_3^2]_\times F^{32} & F^{23}e_3^2\\
        (e_1^2)^\top F^{13}[e_3^2]_\times F^{32} & (e_1^2)^\top F^{13}e_2^3
    \end{bmatrix}.
\end{align}
By the fact that the epipoles in each image coincide and \Cref{eq: collin epip}, this matrix is skew-symmetric which proves the statement.
\end{proof}


The following observation comes directly from the proofs:

\begin{corollary}\label{lem:triple_has_unique_solution}
    Let $\Set{F^{12},F^{13},F^{23}}$ be compatible. If $C_1,C_2$ is a solution to $F^{12}$, then there is a unique solution of the third camera $C_3$ only in the Non-Collinear Case. In other words, there is a unique solution for the cameras $C_1,C_2,C_3$, up to $\mathrm{PGL}_4$ action, if and only if the two epipoles in each image are distinct.
\end{corollary}

\section{Geometric Interpretations for Quadruplewise  Constraints}\label{s: Geo int 4}

In this section we give geometric interpretations for quadruplewise conditions. Given a set of four cameras, we characterize the geometry of their camera centers as fitting into one of four cases, and as in \Cref{s: Geo int 3}, this geometry is also captured by the epipoles in each image:
\begin{enumerate}[leftmargin =3.5em]
    \item [Case 1:] The cameras are in generic position, meaning no plane contains all four centers. In each image, the three epipoles are in generic position, meaning they do not lie on a line.
    \item [Case 2:] All camera centers lie in the same plane, but no three lie on a line. In each image, the three epipoles are distinct and lie on a line.
    \item [Case 3:] Precisely three camera centers lie on a line. In the three corresponding images, the epipoles corresponding to the other two cameras among this triplet are equal, with the last one different from these two. In the final image, the three epipoles are distinct and lie on a line. 
    \item [Case 4:] All four camera centers lie on a line. In each image, the three epipoles coincide.
\end{enumerate}

In \cite{braatelund2023compatibility}, quadruplewise compatibility conditions were characterized for each of these cases. Their proof strategy was to firstly simplify the fundamental matrices via the fundamental action, and secondly to deduce possible reconstructions of cameras for these simplified matrices, and to use this information to deduce a necessary and sufficient conditions. We give new proofs for these conditions based on geometric arguments. 

\subsection{Case 1}

Recall the definition of the epipolar numbers $\textbf{e}_{sijt}=(e_i^s)^\top F^{ij}e_j^t$.

\begin{theorem}[Case 1] 
\label{thm: 4tuple-condition} Let $\Set{F^{ij}}$ be a six-tuple of fundamental matrices such that the three epipoles in each image do not lie on a line. Then $\Set{F^{ij}}$ is compatible if and only if the triplewise conditions hold and
\begin{align}
\begin{aligned} \label{eq: 4-tuple}
\textnormal{\textbf{e}}_{4123}\textnormal{\textbf{e}}_{2134}\textnormal{\textbf{e}}_{3142}\textnormal{\textbf{e}}_{4231}\textnormal{\textbf{e}}_{1243}\textnormal{\textbf{e}}_{2341}=\textnormal{\textbf{e}}_{3124}\textnormal{\textbf{e}}_{4132}\textnormal{\textbf{e}}_{2143}\textnormal{\textbf{e}}_{1234}\textnormal{\textbf{e}}_{3241}\textnormal{\textbf{e}}_{1342}.
\end{aligned}
\end{align}

Further, \Cref{eq: 4-tuple} is equivalent to the condition that given a reconstruction $C_1,C_2,C_3$ of the first three cameras and any point $X\in\PP^3$, the three epipolar lines $F^{41}C_1X, F^{42}C_2X, F^{43}C_3X$ have a unique shared intersection.
\end{theorem}

We refer to the latter condition as the \textit{epipolar line condition}. For the sake of readability, we state and prove two lemmas before we prove \Cref{thm: 4tuple-condition}. Below, we work with affine representatives of all fundamental matrices, cameras, centers and epipoles, and we write $=$ to denote affine equality and $\sim$ to denote projective equality (equality up to non-zero scaling).

\begin{lemma}\label{lem:case 1 geometry}
Let $\Set{F^{ij}}$ be a six-tuple of fundamental matrices such that the three epipoles in each image do not lie on a line and the triplewise conditions hold. Suppose that $C_1,C_2,C_3$ is a reconstruction of the first three cameras and let $p_1,p_2,p_3$ denote their centers. The fourth camera center $p_4$ can be recovered, i.e. the back-projected lines $C_1^{-1}e_1^4,C_2^{-1}e_2^4,C_3^{-1}e_3^4$ meet in a unique point away from the plane spanned by $p_1,p_2,p_3$.  

If $X\in\PP^3$ is any point projectively independent of $p_1,p_2,p_3,p_4$, then $\Set{F^{ij}}$ is compatible if and only if the three epipolar lines $F^{41}C_1X, F^{42}C_2X, F^{43}C_3X$ have a unique shared intersection. 
\end{lemma}

\begin{proof} We have by \Cref{thm: non colin}, that for distinct $i,j\in \{1,2,3\}$, $C_i^{-1}e_i^4$ meets $C_j^{-1}e_j^4$ in a unique point $p_{ij}\in \PP^3$ such that $p_i,p_j,p_{ij}$ are non-collinear. We note that each $p_{ij}$ must lie outside the span $P$ of $p_1,p_2,p_3$, since if $p_{ij}\in P$ then $e_i^4=C_ip_{ij}$ would be collinear with $e_i^j$ and $e_i^k$, for $i,j,k$ distinct among $\{1,2,3\}$, which would contradict the fact that the epipoles are non-collinear in each image.

Since $p_{ij}$ lies outside $P$, the back-projected line $C_i^{-1}e_i^4$ cannot lie inside $P$. Then there is a unique point $p_4$ meeting each $C_i^{-1}e_i^4$ away from $P$ by the following fact: If three lines $L_1,L_2,L_3$ go through fixed non-collinear points $a_1,a_2,a_3\in \PP^3$ and meet pairwise, then they must have a unique common intersect unless some $L_i$ lies in the plane spanned by $a_i$.

$\Rightarrow)$ If $\Set{F^{ij}}$ is compatible, then there exists a camera $C_4$ such that $\{C_i\}_{i=1}^4$ has the appropriate fundamental matrices. In particular, $(C_4X)^\top F^{4i}C_iX=0$ for all $i=1,2,3$. The intersection point $C_4X$ of the lines defined by $(F^{4i}C_iX)^\top$ must be unique, because otherwise all $F^{4i}C_iX$ coincide, but the epipoles $e_4^i$ are not collinear. 

$\Leftarrow)$ Suppose the lines $F^{41}C_1X, F^{42}C_2X, F^{43}C_3X$ intersect uniquely in some point $x_4$, as in \cref{fig:lines}. 
Let $i,j\in \{1,2,3\}$ be distinct. Given the two cameras $C_i,C_j$, triplewise compatibility implies that there exist unique cameras $C_4^{ij}$, with center $p_4$, such that their fundamental matrices are $F^{ij},F^{i4},F^{j4}$. It suffices to show that $C_4^{12}=C_4^{13}=C_4^{23}$, because then by construction $F^{ij}$, for $i,j\in \{1,2,3,4\}$, are the fundamental matrices to $C_1,C_2,C_3$ and (up to scaling) $C_4:=C_4^{12}=C_4^{13}=C_4^{23}$. The cameras $C_4^{ij}$ satisfy $C_4^{ij}p_4=0$ and $C_4^{ij}p_k\sim e_4^k$ for $k=i,j$ by construction. Next, $C_4^{ij}p_k\sim e_4^k$ for $k\in \{1,2,3\}\setminus\{i,j\}$ by the triplewise conditions and $C_4^{ij}X\sim x_4$ by the epipolar line condition. 
Up to change of coordinates, we may assume $p_1=(1,0,0,0),\ldots,p_4=(0,0,0,1)$ and $X=(1,1,1,1)$. These make a projective frame of $\PP^3$ and there can only be one matrix $\widetilde{C}_4$ satisfying these equalities; all $C_4^{ij}$ are equal up to scaling. 
\end{proof}

\begin{figure}\centering
\fbox{%
\begin{picture}(5,5)
\put(1,3){{\circle*{0.1}}}
\put(0.9,3.2){\hbox{\kern3pt \footnotesize\texttt{$e_4^1$}}}
\put(2,0.75){{\circle*{0.1}}\hbox{\kern3pt \footnotesize\texttt{$e_4^2$}}}
\put(3.5,1.5){{\circle*{0.1}}\hbox{\kern3pt \footnotesize\texttt{$e_4^3$}}}
\thinlines
\put(1,3){\line(1,0.2){4}}
\put(1,3){\line(-1,-0.2){1}}
\put(2,0.75){\line(1,2){2.125}}
\put(2,0.75){\line(-1,-2){0.375}}
\put(3.5,1.5){\line(1,-0.5){1.5}}
\put(3.5,1.5){\line(-1,0.5){3.5}}
\end{picture}}
\quad
\fbox{%
\begin{picture}(5,5)
\put(1,3){{\circle*{0.1}}}
\put(0.9,3.2){\hbox{\kern3pt \footnotesize\texttt{$e_4^1$}}}
\put(2,0.75){{\circle*{0.1}}\hbox{\kern3pt \footnotesize\texttt{$e_4^2$}}}
\put(3.5,1.5){{\circle*{0.1}}\hbox{\kern3pt \footnotesize\texttt{$e_4^3$}}}
\put(2.642,3.32){{\circle*{0.1}}}
\put(2.644,3.15){\hbox{\kern3pt \footnotesize\texttt{$x_4$}}}
\thinlines
\put(1,3){\line(1,0.2){4}}
\put(1,3){\line(-1,-0.2){1}}
\put(2,0.75){\line(0.5,2){1.065}}
\put(2,0.75){\line(-0.5,-2){0.19}}
\put(3.5,1.5){\line(0.235,-0.5){0.71}}
\put(3.5,1.5){\line(-0.235,0.5){1.65}}
\end{picture}}
\caption{Left: Generically, the three epipolar lines have no common intersection. Right: If the three epipolar lines share a common intersection $x_4$, then there exists a camera $C_4$ yielding the desired fundamental matrices.}\label{fig:lines}
\end{figure}
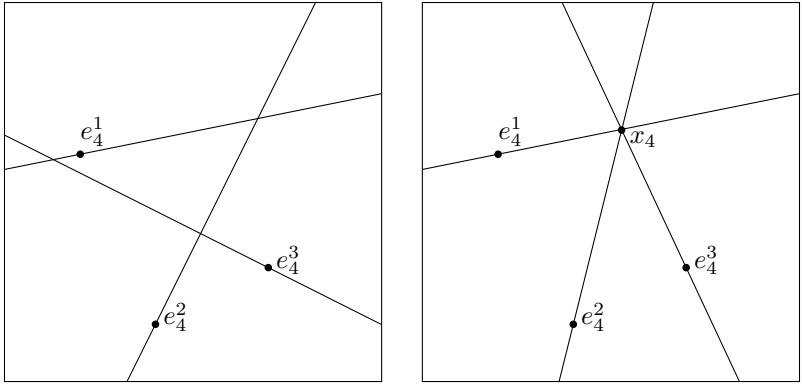

This statement, in conjunction with the following Lemma, will be enough to prove \Cref{thm: 4tuple-condition}.
\begin{lemma}\label{le: C1C2C3} Let $\Set{F^{ij}}$ be a six-tuple of fundamental matrices such that the three epipoles in each image do not lie on a line and the triplewise conditions hold. We can reconstruct the first three cameras as
    \begin{equation}\label{eq: C1C2C3}
\begin{split}
&C_1=\begin{bmatrix}
    \vert & \vert & \vert & \vert \\
    0 & \textnormal{\textbf{e}}_{3124}e_1^2 & -\textnormal{\textbf{e}}_{4123}e_1^3 & \textnormal{\textbf{e}}_{3124}e_1^4\\
    \vert & \vert & \vert & \vert
\end{bmatrix},\quad\quad
C_2=\begin{bmatrix}
    \vert & \vert & \vert & \vert \\
    -\textnormal{\textbf{e}}_{4231}e_2^1 & 0 & \textnormal{\textbf{e}}_{1234}e_2^3 & \textnormal{\textbf{e}}_{1234}e_2^4\\
    \vert & \vert & \vert & \vert
\end{bmatrix},\\
&C_3=\begin{bmatrix}
    \vert & \vert & \vert & \vert \\
    \textnormal{\textbf{e}}_{4132}e_3^1 & -\textnormal{\textbf{e}}_{2134}e_3^2 & 0 & \textnormal{\textbf{e}}_{4132}e_3^4\\
    \vert & \vert & \vert & \vert
\end{bmatrix}.
\end{split}
\end{equation}
For these cameras, $p_4=(0,0,0,1)$ is the unique point satisfying with $C_ip_4\sim e_i^4$.
\end{lemma}
\begin{proof} Firstly, we note that each epipolar number appearing in these matrices are non-zero. This can be seen by the triplewise conditions, $(e_i^s)^\top F^{ij}e_j^t=(e_i^s)^\top (e_i^j\times e_i^t)$. In particular, the matrices are all full-rank.

It can be verified that these cameras have the appropriate fundamental matrices. To do so, let $p=(a,b,c,d)\in\PP^3$ be any arbitrary point and first consider the cameras $C_1,C_2$. Then
\begin{equation}
\begin{split}
(C_1p)^\top F^{12}C_2p&=(-c\textnormal{\textbf{e}}_{4123}e_1^3+d\textnormal{\textbf{e}}_{3124}e_1^4)^\top F^{12}(c\textnormal{\textbf{e}}_{1234}e_2^3+d\textnormal{\textbf{e}}_{1234}e_2^4)\\
&=-cd\textnormal{\textbf{e}}_{4123}\textnormal{\textbf{e}}_{1234}\textnormal{\textbf{e}}_{3124}+cd\textnormal{\textbf{e}}_{3124}\textnormal{\textbf{e}}_{1234}\textnormal{\textbf{e}}_{4123}\\
&=0.
\end{split}
\end{equation}
Similarly, for the cameras $C_1,C_3$ we have
\begin{equation}
\begin{split}
(C_1p)^\top F^{13}C_3p&=(b\textnormal{\textbf{e}}_{3124}e_1^2+d\textnormal{\textbf{e}}_{3124}e_1^4)^\top F^{13}(-b\textnormal{\textbf{e}}_{2134}e_3^2+d\textnormal{\textbf{e}}_{4132}e_3^4)\\
&=b\textnormal{\textbf{e}}_{3124}\textnormal{\textbf{e}}_{4132}\textnormal{\textbf{e}}_{2134}-bd\textnormal{\textbf{e}}_{3124}\textnormal{\textbf{e}}_{2134}\textnormal{\textbf{e}}_{4132}\\
&=0.
\end{split}
\end{equation}
Finally, for the cameras $C_2,C_3$ it holds that
\begin{equation}
\begin{split}
(C_2p)^\top F^{23}C_3p&=(-a\textnormal{\textbf{e}}_{4231}e_2^1+d\textnormal{\textbf{e}}_{1234}e_2^4)^\top F^{23}(a\textnormal{\textbf{e}}_{4132}e_3^1+d\textnormal{\textbf{e}}_{4132}e_3^4)\\
&=-ad\textnormal{\textbf{e}}_{4231}\textnormal{\textbf{e}}_{4132}\textnormal{\textbf{e}}_{1234}+ad\textnormal{\textbf{e}}_{1234}\textnormal{\textbf{e}}_{4132}\textnormal{\textbf{e}}_{4231}\\
&=0.
\end{split}
\end{equation}
Therefore, these cameras are valid choices for reconstruction.

The fact that there is such a unique point $p_4$ follows from \Cref{lem:case 1 geometry}, and by inspection is must be the point $(0,0,0,1)$.
\end{proof}

\begin{proof}[Proof of \Cref{thm: 4tuple-condition}] We use \Cref{le: C1C2C3} to assume that $C_1,C_2,C_3$ take the form of \Cref{eq: C1C2C3}, and note that $X=(1,1,1,1)$ is projectively independent of $p_1,p_2,p_3,p_4$. By \Cref{lem:case 1 geometry} the six-tuple $\Set{F^{ij}}$ is compatible if and only if the lines $F^{41}C_1X$, $F^{42}C_2X$ and $F^{43}C_3X$ have a shared intersection. Equivalently, the six-tuple is compatible if and only if
\begin{equation}
\begin{vmatrix}
    \vert & \vert & \vert\\
    F^{41}C_1X & F^{42}C_2X & F^{43}C_3X\\
    \vert & \vert & \vert
\end{vmatrix}=0
\end{equation}
Because the $e_4^i$ are not collinear, we can multiply the determinant of $\begin{bmatrix}
   e_4^1 & e_4^2 & e_4^3 
\end{bmatrix}^\top$ to get the equivalent condition
\begin{equation}
\begin{vmatrix}
    (e_4^1)^\top F^{41}C_1X & (e_4^1)^\top F^{42}C_2X & (e_4^1)^\top F^{43}C_3X\\
    (e_4^2)^\top F^{41}C_1X & (e_4^2)^\top F^{42}C_2X & (e_4^2)^\top F^{43}C_3X\\
    (e_4^3)^\top F^{41}C_1X & (e_4^3)^\top F^{42}C_2X & (e_4^3)^\top F^{43}C_3X
\end{vmatrix}=0
\end{equation}
This simplifies to
\begin{equation}
\begin{split}
&\begin{vmatrix}
    0 & (e_4^1)^\top F^{42}\textnormal{\textbf{e}}_{1234}e_2^3 & -(e_4^1)^\top F^{43}\textnormal{\textbf{e}}_{2134}e_3^2\\
    -(e_4^2)^\top F^{41}\textnormal{\textbf{e}}_{4123}e_1^3 & 0 & (e_4^2)^\top F^{43}\textnormal{\textbf{e}}_{4132}e_3^1\\
    (e_4^3)^\top F^{41}\textnormal{\textbf{e}}_{3124}e_1^2 & -(e_4^3)^\top F^{42}\textnormal{\textbf{e}}_{4231}e_2^1 & 0
\end{vmatrix}\\
=&\begin{vmatrix}
    0 & \textnormal{\textbf{e}}_{1423}\textnormal{\textbf{e}}_{1234} & -\textnormal{\textbf{e}}_{1432}\textnormal{\textbf{e}}_{2134}\\
    -\textnormal{\textbf{e}}_{2413}\textnormal{\textbf{e}}_{4123} & 0 & \textnormal{\textbf{e}}_{2431}\textnormal{\textbf{e}}_{4132}\\
    \textnormal{\textbf{e}}_{3412}\textnormal{\textbf{e}}_{3124} & -\textnormal{\textbf{e}}_{3421}\textnormal{\textbf{e}}_{4231} & 0
\end{vmatrix}\\
=&0
\end{split}
\end{equation}
which can be seen to be exactly \Cref{eq: 4-tuple}.
\end{proof}

\subsection{Case 2}

\begin{theorem}[Case 2]\label{thm: Case 2 general form} Let $\Set{F^{ij}}$ be a six-tuple of fundamental matrices whose epipoles in each image are distinct and lie on a line. Let $e_i^5$ be arbitrary points linearly independent of the other epipoles in the $i$-th image. Then $\Set{F^{ij}}$ is compatible if and only if the triplewise conditions hold,
\begin{equation}\label{eq:case 2 short}\begin{split}
\textnormal{\textbf{e}}_{ijk5}\textnormal{\textbf{e}}_{ikl5}\textnormal{\textbf{e}}_{ilj5}+\textnormal{\textbf{e}}_{ikj5}\textnormal{\textbf{e}}_{ijl5}\textnormal{\textbf{e}}_{ilk5}=0,
\end{split}\end{equation}
for each $i$ and $j<k<l$, and
    \begin{equation}\label{eq:case 2 long}\begin{split}
&\textnormal{\textbf{e}}_{3245}\textnormal{\textbf{e}}_{2315}\textnormal{\textbf{e}}_{2415}\textnormal{\textbf{e}}_{1325}\textnormal{\textbf{e}}_{1435}\textnormal{\textbf{e}}_{5125}-\textnormal{\textbf{e}}_{2345}\textnormal{\textbf{e}}_{3215}\textnormal{\textbf{e}}_{2415}\textnormal{\textbf{e}}_{1325}\textnormal{\textbf{e}}_{1425}\textnormal{\textbf{e}}_{5135}-\textnormal{\textbf{e}}_{3215}\textnormal{\textbf{e}}_{2315}\textnormal{\textbf{e}}_{1325}\textnormal{\textbf{e}}_{1425}\textnormal{\textbf{e}}_{2435}\textnormal{\textbf{e}}_{5145}\\
-&\textnormal{\textbf{e}}_{1345}\textnormal{\textbf{e}}_{3215}\textnormal{\textbf{e}}_{2315}\textnormal{\textbf{e}}_{2415}\textnormal{\textbf{e}}_{1425}\textnormal{\textbf{e}}_{5235}-\textnormal{\textbf{e}}_{3215}\textnormal{\textbf{e}}_{2315}\textnormal{\textbf{e}}_{2415}\textnormal{\textbf{e}}_{1325}\textnormal{\textbf{e}}_{1435}\textnormal{\textbf{e}}_{5245}+\textnormal{\textbf{e}}_{3215}\textnormal{\textbf{e}}_{2315}\textnormal{\textbf{e}}_{2415}\textnormal{\textbf{e}}_{1325}\textnormal{\textbf{e}}_{1425}\textnormal{\textbf{e}}_{5345}=0.
\end{split}\end{equation}
The points $e_i^5$ may be chosen independently for each of the five equations.

Further, \Cref{eq:case 2 short} is equivalent to the condition that after reconstructing the cameras $C_i,C_j,C_k$ the camera center $p_l$ can be recovered, i.e. the back-projected lines $C_i^{-1}e_i^l,C_j^{-1}e_j^l,C_k^{-1}e_k^l$ have a unique shared intersection, and \Cref{eq:case 2 long} is equivalent to the epipolar line condition from \Cref{thm: 4tuple-condition}.
\end{theorem}

\begin{remark}
If we specify $e_1^5=F^{12}e_2^4,e_2^5=F^{23}e_3^4$ and $e_3^5=F^{31}e_1^4$ then we obtain exactly equation (49) from \textnormal{\cite{braatelund2023compatibility}} for $i=4$.
\end{remark}

We note that, in contrast to Case 1, in Case 2, the reconstructability of the $4$-th camera center is not guaranteed. To compare with \Cref{lem:case 1 geometry}, the three back-projected lines $C_1^{-1}e_1^4,C_2^{-1}e_2^4,C_3^{-1}e_3^4$ are coplanar, and there is a priori no unique intersection point.  

Before we prove \Cref{thm: Case 2 general form}, we need three lemmas. 

\begin{lemma}\label{le: back proj lines Case 2} Let $\Set{F^{ij}}$ be a six-tuple of fundamental matrices whose epipoles in each image are distinct and lie on a line. Suppose that $C_1,C_2,C_3$ is a reconstruction of the first three cameras, and define $e_i^5$ to be arbitrary points linearly independent of the other epipoles in the $i$-th image. The fourth camera center $p_4$ can be reconstructed, i.e. $C_1^{-1}e_1^4,C_2^{-1}e_2^4,C_3^{-1}e_3^4$ have a unique intersection point if and only if \Cref{eq:case 2 short} holds for $i=4$: 
\begin{align}\label{eq: back proj plane}
\textnormal{\textbf{e}}_{4325}\textnormal{\textbf{e}}_{4135}\textnormal{\textbf{e}}_{4215}+\textnormal{\textbf{e}}_{4315}\textnormal{\textbf{e}}_{4125}\textnormal{\textbf{e}}_{4235}=0.
\end{align}
\end{lemma}

\begin{proof} Let $C_5$ be any camera with center $p_5$ away from the span of $p_1,p_2,p_3$ and define $x_i:=C_ip_5$ for $i=1,2,3$. After projective transformation in each image, we can assume that $x_i=e_i^5$, and moreover that $e_i^j,e_i^k,e_i^5$ are equal to $(1,0,0),(0,1,0),(0,0,1)$, respectively. Therefore, $e_i^4=(a_i,b_i,0)$ for some $a_i,b_i\neq 0$. Then, as in \Cref{le: C1C2C3}, we may assume that $C_1,C_2,C_3$ are equal to 
  \begin{equation}
\begin{split}
&C_1=\begin{bmatrix}
    0 & \textnormal{\textbf{e}}_{3125} & 0 & 0 \\
    0 & 0 & -\textnormal{\textbf{e}}_{5123} & 0\\
    0 & 0 & 0 & \textnormal{\textbf{e}}_{3125}
\end{bmatrix},\quad\quad
C_2=\begin{bmatrix}
    -\textnormal{\textbf{e}}_{5231} & 0 & 0 & 0 \\
    0 & 0 & \textnormal{\textbf{e}}_{1235} &0\\
    0 & 0 & 0 &\textnormal{\textbf{e}}_{1235}
\end{bmatrix},\\
&C_3=\begin{bmatrix}
    \textnormal{\textbf{e}}_{5132} & 0 & 0 & 0 \\
    0 & -\textnormal{\textbf{e}}_{2135} & 0 & 0\\
    0 & 0 & 0 & \textnormal{\textbf{e}}_{5132}
\end{bmatrix}.
\end{split}
\end{equation}
The back-projected lines of $e_i^4$ are now easy to describe using pseudo-inverses of each camera. For instance,
\begin{align}
    \mu \begin{bmatrix}
        1 \\ 0\\ 0\\ 0
    \end{bmatrix} + \lambda  \begin{bmatrix}
        0& 0 & 0 \\ 1/\textnormal{\textbf{e}}_{3125} & 0 & 0\\ 0 & -1/\textnormal{\textbf{e}}_{5123}& 0\\ 0 & 0 & 1/\textnormal{\textbf{e}}_{3125}
    \end{bmatrix} e_1^4,
\end{align}
is a parametrization of $C_1^{-1}e_1^4$. All back-projected lines $C_i^{-1}e_i^4$, lie in the hyperplane of $\PP^3$ with last coordinate $0$, as a consequence of $e_i^4=(a_i,b_i,0)$. In this copy of $\PP^2$, the lines are defined by the vectors
\begin{align}
    \begin{bmatrix}
        1 \\ 0\\ 0
    \end{bmatrix}\times  \begin{bmatrix}
        0 \\ a_1/\textnormal{\textbf{e}}_{3125} \\ -b_1/\textnormal{\textbf{e}}_{5123}
    \end{bmatrix},\quad  \begin{bmatrix}
        0 \\ 1\\ 0
    \end{bmatrix}\times  \begin{bmatrix}
        -a_2/\textnormal{\textbf{e}}_{5231}\\ 0 \\ b_2/\textnormal{\textbf{e}}_{1235}
    \end{bmatrix}, \quad  \begin{bmatrix}
        0 \\ 0\\ 1
    \end{bmatrix}\times  \begin{bmatrix}
        a_3/\textnormal{\textbf{e}}_{5132} \\ -b_3/\textnormal{\textbf{e}}_{2135} \\ 0
    \end{bmatrix}.
\end{align}
Computing these vectors, and concatenating them into a $3\times 3$ matrix $M$, we have by $\textnormal{\textbf{e}}_{sijt}=\textnormal{\textbf{e}}_{tjis}$,
\begin{align}\label{eq: mat}
    M=\begin{bmatrix}
        0 & b_2\textnormal{\textbf{e}}_{5231} & b_3\textnormal{\textbf{e}}_{5132} \\
        b_1\textnormal{\textbf{e}}_{3125} & 0 & a_3\textnormal{\textbf{e}}_{2135}\\
        a_1\textnormal{\textbf{e}}_{5123} & a_2\textnormal{\textbf{e}}_{1235} & 0
    \end{bmatrix}=\begin{bmatrix}
        0 & b_2\textnormal{\textbf{e}}_{1325} & b_3\textnormal{\textbf{e}}_{2315} \\
        b_1\textnormal{\textbf{e}}_{3125} & 0 & a_3\textnormal{\textbf{e}}_{2135}\\
        a_1\textnormal{\textbf{e}}_{3215} & a_2\textnormal{\textbf{e}}_{1235} & 0
    \end{bmatrix} .
\end{align}
The determinant of this matrix is 0 if and only if there is a point meeting the three back-projected lines. Next, we use $e_i^4=a_ie_i^j+b_ie_i^k$ for distinct $i,j,k$ with $j<k$ to deduce that
\begin{align}
    (e_i^k)^\top F^{ij}e_j^5=\frac{1}{b_i}(e_i^4)^\top F^{ij}e_j^5,\quad    (e_i^j)^\top F^{ik}e_k^5=\frac{1}{a_i}(e_i^4)^\top F^{ik}e_k^5.
\end{align}
Then 
\begin{align}
    M=\begin{bmatrix}
        0 & \frac{b_2}{a_3}\textnormal{\textbf{e}}_{4325} & \textnormal{\textbf{e}}_{4315} \\
        \textnormal{\textbf{e}}_{4125} & 0 & \frac{a_3}{a_1}\textnormal{\textbf{e}}_{4135}\\
        \frac{a_1}{b_2}\textnormal{\textbf{e}}_{4215} & \textnormal{\textbf{e}}_{4235} & 0
    \end{bmatrix} .
\end{align}
We can now see that $\det M=0$ is \Cref{eq: back proj plane}. 

The intersection point must finally be unique, because the three back-projected lines cannot meet in a line; the center are not collinear. 
\end{proof}

We now show that condition \Cref{eq:case 2 short} together with the epipolar line condition are necessary and sufficient for compatibility.
\begin{lemma}\label{le: shared int Case 2}  Let $\Set{F^{ij}}$ be a six-tuple of fundamental matrices whose epipoles in each image are distinct and lie on a line. Then $\Set{F^{ij}}$ is compatible if and only if \Cref{eq:case 2 short} holds for all $i=1,2,3,4$ and the epipolar line condition holds. That is, after reconstructing $C_1,C_2,C_3$ and the camera center $p_4$, for any point $p_5$ away from the plane spanned by $p_1,p_2,p_3,p_4$, the three epipolar lines $F^{41}C_1p_5, F^{42}C_2p_5, F^{43}C_3p_5$ have a unique shared intersection. 
\end{lemma}

Here we use the notation $p_5$ instead of $X$ as in \Cref{lem:case 1 geometry}, because we denote the images of $p_5$ with respect to the cameras by $e_i^5$, notation used in \Cref{thm: Case 2 general form}.

\begin{proof} It follows by \Cref{le: back proj lines Case 2} that the first four conditions are necessary for compatibility. Moreover, assuming compatibility holds and there exists a reconstruction $C_1,C_2,C_3,C_4$, then $C_4p_5$ must be the unique point in the intersection of the epipolar lines.    


For the other direction, suppose the lines $F^{41}C_1p_5, F^{42}C_2p_5, F^{43}C_3p_5$ intersect uniquely in some point $e_4^5$. Let $i,j\in \{1,2,3\}$ be distinct. Given the two cameras $C_i,C_j$, triplewise compatibility implies that there exist unique cameras $C_4^{ij}$, with center $p_4$, such that their fundamental matrices are $F^{ij},F^{i4},F^{j4}$. We are done if we can prove that $C_4^{12},C_4^{13},C_4^{23}$ are all equal, because then by construction $F^{ij}$, for $i,j\in \{1,2,3,4\}$, are the fundamental matrices to $C_1,C_2,C_3$ and (up to scaling) $C_4:=C_4^{12}=C_4^{13}=C_4^{23}$. The cameras $C_4^{ij}$ satisfy $C_4^{ij}p_4=0$ and $C_4^{ij}p_k\sim e_4^k$ for $k=i,j$ by construction. By \Cref{le: back proj lines Case 2}, the back-projected lines $C_i^{-1}e_i^k,$ $C_j^{-1}e_j^k,$ $(C_4^{ij})^{-1}e_4^k$ meet in a unique point for $k\in \{1,2,3\}\setminus\{i,j\}$. This unique point must be $p_3$, because this is the unique intersection point of $C_1^{-1}e_1^3,C_2^{-1}e_2^3$ as in \Cref{thm: non colin}. Then $C_4^{ij}p_k\sim e_4^k$, for $k\in \{1,2,3\}\setminus\{i,j\}$, and by the epipolar line conditions, we have $C_4^{ij}p_5\sim e_4^5$. Define $p_6$ to be any point in the line spanned by $p_4,p_5$, away from both $p_4,p_5$. In contrast to \Cref{lem:case 1 geometry}, $\{p_1,\ldots,p_5\}$ is not a projective frame, although one can then check that $\{p_1,p_2,p_3,p_5,p_6\}$ is. Observe that $C_4^{ij}p_6\sim C_4^{ij}p_5$. 

We next prove that $C_4^{12},C_4^{13}$ are the same up to scaling. Because $C_1,C_4^{12}$ and $C_1,C_4^{13}$ have the same fundamental matrix $F^{14}$, there must exist an invertible projective transformation $H$ such that $C_1H=C_1$ and $C_4^{12}H=C_4^{13}$. Since $C_4^{12},C_4^{13}$ act the same on each $p_i$, we must have  $Hp_i=\lambda_i p_i+\mu_ip_4$, for $i=1,\ldots,6$. However, the fact that $C_1H=C_1$ and $p_1\neq p_4$, implies that $\mu_i=0$ for each $i=1,2,3,5,6$. Then $H$ must be the identity matrix (up to scaling) and thus $C_4^{12}\sim C_4^{13}$. Repeating the argument, all three $C_4^{ij}$ are the same and we are done.
\end{proof}

\begin{lemma}\label{le: G forms c2} Let $\Set{F^{ij}}$ be a six-tuple of fundamental matrices whose epipoles in each image are distinct and lie on a line. If triplewise conditions hold, then up to fundamental action, the fundamental matrices equal 
\begin{align}\label{eq: Gform c2}
\begin{aligned}
    &G^{12}=\begin{bmatrix} 0&0&0\\0&0&x_{12}\\0&y_{12}&z_{12} \end{bmatrix},
    &&G^{13}=\begin{bmatrix} 0&0&x_{13}\\0&0&0\\0&y_{13}&z_{13} \end{bmatrix}, &&G^{14}=\begin{bmatrix} 0&0&x_{14}\\0&0&-x_{14}\\0&y_{14}&z_{14} \end{bmatrix},\\
    &G^{23}=\begin{bmatrix} 0&0&x_{23}\\0&0&0\\y_{23}&0&z_{23} \end{bmatrix}, &&G^{24}=\begin{bmatrix} 0&0&x_{24}\\0&0&-x_{24}\\y_{24}&0&z_{24} \end{bmatrix},
    &&G^{34}=\begin{bmatrix} 0&0&x_{34}\\0&0&-x_{34}\\y_{34}&-y_{34}&z_{34} \end{bmatrix}.
\end{aligned}
\end{align}
Further, if \Cref{eq:case 2 short} holds for $i=4$, then we can reconstruct the first three cameras as
    \begin{align}\label{eq: C rec}
\begin{aligned}
    &C_1=\begin{bmatrix} 0&1&0&0\\0&0&1&0\\0&0&0&x_{12}x_{13}y_{23}\end{bmatrix},  & &C_2=\begin{bmatrix} 1&0&0&0\\0&0&1&z_{12}x_{13}y_{23}\\0&0&0&-y_{12}x_{13}y_{23} \end{bmatrix},\\
    &C_3=\begin{bmatrix} 1&0&0&z_{23}x_{12}y_{13}\\0&1&0&z_{13}x_{12}y_{23}\\0&0&0&-x_{12}y_{13}y_{23} \end{bmatrix}.
\end{aligned}
\end{align}
\end{lemma}

\begin{proof} Fix a scaling of the epipoles such that $e_i^l=e_i^j+e_i^k$ for $j<k<l$. The first part of the statement was observed in \cite{braatelund2023compatibility}. The idea is to apply the fundamental action given by $H_i=\begin{bmatrix} e_i^j & e_i^k & x_i    
\end{bmatrix}$ for distinct $i,j,k,l\in \{1,2,3,4\}$ with $j<k<l$ and some $x_i$ making $H_i$ invertible. Then by the triplewise conditions, $G^{ij}:=H_i^\top F^{ij}H_j$ must be on the form \Cref{eq: Gform c2}.

To see that \Cref{eq: C rec} is a valid reconstruction, let $X=(a,b,c,d)$ and see that
\begin{equation}\begin{split}
(C_1X)^\top G^{12}C_2X&=(b,c,dx_{12}x_{13}y_{23})^\top G^{12}(a,c+dz_{12}x_{13}y_{23},-dy_{12}x_{13}y_{23})\\
&=(b,c,dx_{12}x_{13}y_{23})\cdot(0,-dx_{12}y_{12}x_{13}y_{23},cy_{12})\\
&=0.
\end{split}\end{equation}
Similarly 
\begin{equation}\begin{split}
(C_1X)^\top G^{13}C_3X&=(b,c,dx_{12}x_{13}y_{23})G^{13}(a+dz_{23}x_{12}y_{13},b+dz_{13}x_{12}y_{23},-dx_{12}y_{13}y_{23})\\
&=(b,c,dx_{12}x_{13}y_{23})\cdot(-dx_{13}x_{12}y_{13}y_{23},0,by_{13})\\
&=0,
\end{split}\end{equation}
and finally 
\begin{equation}\begin{split}
(C_2X)^\top G^{23}C_3X&=(a,c+dz_{12}x_{13}y_{23},-dy_{12}x_{13}y_{23})G^{23}(a+dz_{23}x_{12}y_{13},b+dz_{13}x_{12}y_{23},-dx_{12}y_{13}y_{23})\\
&=(a,c+dz_{12}x_{13}y_{23},-dy_{12}x_{13}y_{23})\cdot(-dx_{23}x_{12}y_{13}y_{23},0,ay_{23})\\
&=-ady_{23}(y_{12}x_{13}y_{23}+x_{12}y_{13}x_{23})\\
&=0,
\end{split}\end{equation}
by \Cref{eq:case 2 short} for $i=4$, which finishes the proof. 
 \end{proof}


\begin{proof}[Proof of \Cref{thm: Case 2 general form}] We use \Cref{le: G forms c2} to get simplified fundamental matrices $G^{ij}$, having applied the fundamental action $H_i=\begin{bmatrix} e_i^j & e_i^k & e_i^5
\end{bmatrix}$, for $i=1,2,3,4$ with $j<k$ being the smallest indices distinct from $i$. Recall that $\Set{F^{ij}}$ is compatible if and only if $\Set{G^{ij}}$ is. By assumption, $H_i$ are invertible. Let $h_i^j:=H_i^{-1}e_i^j$ be the epipoles of $G^{ij}$, and let $h_i^5=(0,0,1)$. Write $\textnormal{\textbf{h}}_{sijt}= (h_i^s)^\top G^{ij}h_j^t$ for the epipolar numbers of $G$. We can now deduce that 
\begin{equation}\begin{split}
x_{ij}=(h_i^k)^\top G^{ij}h_j^5=\textnormal{\textbf{h}}_{kij5},\\
y_{ij}=(h_i^5)^\top G^{ij}h_j^k=\textnormal{\textbf{h}}_{5ijk},\\
z_{ij}=(h_i^5)^\top G^{ij}h_j^5=\textnormal{\textbf{h}}_{5ij5},
\end{split}\end{equation}
where $k$ is chosen such that $k=\min\{1,2,3,4~:~k\neq i,j\}$. 

As in \Cref{le: back proj lines Case 2} Taking the point $p_5=(0,0,0,1)$, consider the three epipolar lines $G^{41}C_1p_5,$ $G^{42}C_2p_5,$ $G^{43}C_3p_5$. 
The matrix $\begin{bmatrix} G^{41}C_1p_5 &G^{42}C_2p_5 &G^{43}C_3p_5 \end{bmatrix}$ equals 
\begin{align}\label{eq: simplified}
\begin{bmatrix}
        0&-y_{12}x_{13}y_{23}y_{24}&-x_{12}y_{13}y_{23}y_{34}\\
        x_{12}x_{13}y_{23}y_{14}&0&x_{12}y_{13}y_{23}y_{34}\\
        x_{12}x_{13}y_{23}z_{14}& -z_{12}x_{13}y_{23}x_{24}-y_{12}x_{13}y_{23}z_{24}& z_{23}x_{12}y_{13}x_{34}-z_{13}x_{12}y_{23}x_{34}-x_{12}y_{13}y_{23}z_{34}
    \end{bmatrix}.
\end{align}
Since all $x_{ij},y_{ij}$ are non-zero, by scaling, we can simplify this to
\begin{align}
\begin{bmatrix}
        0 & y_{12}y_{24} & -y_{13}y_{23}y_{34}\\
        y_{14} & 0 & y_{13}y_{23}y_{34}\\
        z_{14} & z_{12}x_{24}+y_{12}z_{24} & z_{23}y_{13}x_{34}-z_{13}y_{23}x_{34}-y_{13}y_{23}z_{34}
    \end{bmatrix}.
\end{align}
By \Cref{le: shared int Case 2}, we want to understand the condition that the three epipolarr lines meet in a unique point. This happens precisely when the determinant of this matrix is zero. Indeed, by the fact that $y_{ij}$ are non-zero, the matrix of \Cref{eq: simplified} cannot be rank 1. This yields the equation
\begin{equation}\begin{split}
&-x_{24}y_{13}y_{14}y_{23}y_{34}z_{12}+x_{34}y_{12}y_{14}y_{23}y_{24}z_{13}+y_{12}y_{13}y_{23}y_{24}y_{34}z_{14}\\
&-x_{34}y_{12}y_{13}y_{14}y_{24}z_{23}-y_{12}y_{13}y_{14}y_{23}y_{34}z_{24}+y_{12}y_{13}y_{14}y_{23}y_{24}z_{34}=0.
\end{split}\end{equation}
Writing this out in terms of the epipolar numbers, we get
\begin{equation}\begin{split}
&-\textnormal{\textbf{h}}_{1245}\textnormal{\textbf{h}}_{2315}\textnormal{\textbf{h}}_{2415}\textnormal{\textbf{h}}_{1325}\textnormal{\textbf{h}}_{1435}\textnormal{\textbf{h}}_{5125}+\textnormal{\textbf{h}}_{1345}\textnormal{\textbf{h}}_{3215}\textnormal{\textbf{h}}_{2415}\textnormal{\textbf{h}}_{1325}\textnormal{\textbf{h}}_{1425}\textnormal{\textbf{h}}_{5135}+\textnormal{\textbf{h}}_{3215}\textnormal{\textbf{h}}_{2315}\textnormal{\textbf{h}}_{1325}\textnormal{\textbf{h}}_{1425}\textnormal{\textbf{h}}_{1435}\textnormal{\textbf{h}}_{5145}\\&-\textnormal{\textbf{h}}_{1345}\textnormal{\textbf{h}}_{3215}\textnormal{\textbf{h}}_{2315}\textnormal{\textbf{h}}_{2415}\textnormal{\textbf{h}}_{1425}\textnormal{\textbf{h}}_{5235}-\textnormal{\textbf{h}}_{3215}\textnormal{\textbf{h}}_{2315}\textnormal{\textbf{h}}_{2415}\textnormal{\textbf{h}}_{1325}\textnormal{\textbf{h}}_{1435}\textnormal{\textbf{h}}_{5245}+\textnormal{\textbf{h}}_{3215}\textnormal{\textbf{h}}_{2315}\textnormal{\textbf{h}}_{2415}\textnormal{\textbf{h}}_{1325}\textnormal{\textbf{h}}_{1425}\textnormal{\textbf{h}}_{5345}=0.
\end{split}\end{equation}
To ensure homogeneity in the $h_i^j$, we recall that our fixed scalings $e_i^l=e_i^j+e_i^k$ for $j<k<l$ among $\{1,2,3,4\}$ imply the identities $\textnormal{\textbf{h}}_{3245}=-\textnormal{\textbf{h}}_{1245}, \textnormal{\textbf{h}}_{1345}=-\textnormal{\textbf{h}}_{2345},$ and $\textnormal{\textbf{h}}_{1435}=-\textnormal{\textbf{h}}_{2435}$. Performing these substitutions in the first two terms, we have
\begin{equation}\begin{split}
&\textnormal{\textbf{h}}_{3245}\textnormal{\textbf{h}}_{2315}\textnormal{\textbf{h}}_{2415}\textnormal{\textbf{h}}_{1325}\textnormal{\textbf{h}}_{1435}\textnormal{\textbf{h}}_{5125}-\textnormal{\textbf{h}}_{2345}\textnormal{\textbf{h}}_{3215}\textnormal{\textbf{h}}_{2415}\textnormal{\textbf{h}}_{1325}\textnormal{\textbf{h}}_{1425}\textnormal{\textbf{h}}_{5135}-\textnormal{\textbf{h}}_{3215}\textnormal{\textbf{h}}_{2315}\textnormal{\textbf{h}}_{1325}\textnormal{\textbf{h}}_{1425}\textnormal{\textbf{h}}_{2435}\textnormal{\textbf{h}}_{5145}\\
-&\textnormal{\textbf{h}}_{1345}\textnormal{\textbf{h}}_{3215}\textnormal{\textbf{h}}_{2315}\textnormal{\textbf{h}}_{2415}\textnormal{\textbf{h}}_{1425}\textnormal{\textbf{h}}_{5235}-\textnormal{\textbf{h}}_{3215}\textnormal{\textbf{h}}_{2315}\textnormal{\textbf{h}}_{2415}\textnormal{\textbf{h}}_{1325}\textnormal{\textbf{h}}_{1435}\textnormal{\textbf{h}}_{5245}+\textnormal{\textbf{h}}_{3215}\textnormal{\textbf{h}}_{2315}\textnormal{\textbf{h}}_{2415}\textnormal{\textbf{h}}_{1325}\textnormal{\textbf{h}}_{1425}\textnormal{\textbf{h}}_{5345}=0.
\end{split}\end{equation}
Finally, undoing the fundamental action by using that $\textnormal{\textbf{h}}_{sijt}=\textnormal{\textbf{e}}_{sijt}$, we obtain the equation
\begin{equation}\begin{split}
&\textnormal{\textbf{e}}_{3245}\textnormal{\textbf{e}}_{2315}\textnormal{\textbf{e}}_{2415}\textnormal{\textbf{e}}_{1325}\textnormal{\textbf{e}}_{1435}\textnormal{\textbf{e}}_{5125}-\textnormal{\textbf{e}}_{2345}\textnormal{\textbf{e}}_{3215}\textnormal{\textbf{e}}_{2415}\textnormal{\textbf{e}}_{1325}\textnormal{\textbf{e}}_{1425}\textnormal{\textbf{e}}_{5135}-\textnormal{\textbf{e}}_{3215}\textnormal{\textbf{e}}_{2315}\textnormal{\textbf{e}}_{1325}\textnormal{\textbf{e}}_{1425}\textnormal{\textbf{e}}_{2435}\textnormal{\textbf{e}}_{5145}\\
-&\textnormal{\textbf{e}}_{1345}\textnormal{\textbf{e}}_{3215}\textnormal{\textbf{e}}_{2315}\textnormal{\textbf{e}}_{2415}\textnormal{\textbf{e}}_{1425}\textnormal{\textbf{e}}_{5235}-\textnormal{\textbf{e}}_{3215}\textnormal{\textbf{e}}_{2315}\textnormal{\textbf{e}}_{2415}\textnormal{\textbf{e}}_{1325}\textnormal{\textbf{e}}_{1435}\textnormal{\textbf{e}}_{5245}+\textnormal{\textbf{e}}_{3215}\textnormal{\textbf{e}}_{2315}\textnormal{\textbf{e}}_{2415}\textnormal{\textbf{e}}_{1325}\textnormal{\textbf{e}}_{1425}\textnormal{\textbf{e}}_{5345}=0,
\end{split}\end{equation}
which finishes the proof.
\end{proof}
\begin{remark} If $e_i^5$ in \Cref{thm: Case 2 general form} are chosen such that $(e_i^5)^\top F^{ij}e_j^5=0$ for $i,j$ among $1,2,3$, then \Cref{eq:case 2 long} takes the simpler form \begin{equation}\label{eq: simpler case 2}
    \textnormal{\textbf{e}}_{1425}\textnormal{\textbf{e}}_{2435}\textnormal{\textbf{e}}_{5415}+\textnormal{\textbf{e}}_{1435}\textnormal{\textbf{e}}_{2415}\textnormal{\textbf{e}}_{5425}-\textnormal{\textbf{e}}_{1425}\textnormal{\textbf{e}}_{2415}\textnormal{\textbf{e}}_{5435}=0.
\end{equation}
\end{remark}

Under the conditions of this remark, we next give a simpler proof of \Cref{thm: Case 2 general form}. After reconstructing the first three cameras $C_1,C_2,C_3$, choose a generic camera $C_5$ and let $e_i^5$ be the epipoles $C_ip_5$. Then we may assume $C_1,C_2,C_3$ to equal \Cref{eq: C1C2C3}, and we consider by \Cref{le: shared int Case 2} the condition
\begin{align}\label{eq: shared int pf}
    \det \begin{bmatrix}
        F^{41}C_1p_5 & F^{42}C_2p_5 & F^{43}C_3p_5
    \end{bmatrix} =0,
\end{align}
with respect to the center $p_5=(0,0,0,1)$. Writing it out, we have
\begin{align}
    \det \begin{bmatrix}
       \textnormal{\textbf{e}}_{3125} \cdot F^{41}e_1^5  &
       \textnormal{\textbf{e}}_{1235}\cdot   F^{42}e_2^5   &\textnormal{\textbf{e}}_{5132}\cdot  F^{43}e_3^5   
    \end{bmatrix} =0.
\end{align}
We may simplify this equation by scaling in each column, giving us $\det  \begin{bmatrix}
        F^{41}e_1^5  &
         F^{42}e_2^5 & F^{43}e_3^5 
    \end{bmatrix} =0.$ Recall that we have performed coordinate change in each image such that $e_i^j=(1,0,0),e_i^k=(0,1,0)$ and $e_i^5=(0,0,1)$ for $j<k$. With respect to these epipoles, we can have
    \begin{align}
        F^{41}e_1^5=(0,b_1,c_1)^\top, \quad F^{42}e_2^5=(a_2,b_2,0)^\top,\quad F^{43}e_3^5=(a_3,b_3,c_3)^\top,
    \end{align}
    for some constants $e_i,b_i,c_i$. Further, $(e_4^j)^\top F^{4i}e_i^5=a_i,(e_4^k)^\top F^{4i}e_i^5=b_i$ and $(e_4^5)^\top F^{4i}e_i^5=c_i$ for $j<k$ and $i,j,k\in \{1,2,3\}.$ Now, \Cref{eq: shared int pf} becomes
\begin{align}
    \det \begin{bmatrix}
        0 &  \textnormal{\textbf{e}}_{1425} & \textnormal{\textbf{e}}_{1435} \\
        \textnormal{\textbf{e}}_{2415} & 0 & \textnormal{\textbf{e}}_{2435}\\
        \textnormal{\textbf{e}}_{5415} & \textnormal{\textbf{e}}_{5425} & \textnormal{\textbf{e}}_{5435}
    \end{bmatrix} =0,
\end{align}
which is \Cref{eq: simpler case 2}.


\subsection{Case 3} 

\begin{theorem}[Case 3] Let $\Set{F^{ij}}$ be a six-tuple of fundamental matrices such that two of the epipoles in the first three images are equal and are distinct from the third epipole $e_i^4$, and that in the fourth image the three epipoles are distinct but collinear. Then $\Set{F^{ij}}$ is compatible if and only if the triplewise conditions hold.  
\end{theorem}

We clarify that triplewise conditions here means collinear triplewise conditions for $F^{12},F^{13},F^{23}$ and non-collinear triplewise conditions for the other triplets. 

\begin{proof} If $\Set{F^{ij}}$ is compatible, then any subset is compatible, proving one direction.

For the other, let $C_2,C_3,C_4$ be a reconstruction of the three last cameras. Then we argue that $C_2^{-1}e_2^1,C_3^{-1}e_3^1$ and $C_4^{-1}e_4^1$ meet in exactly a point $p_1$ in the span of $p_2,p_3$, but distinct from each $p_2,p_3$. Since $e_2^1\sim e_2^3$ and $e_3^1\sim e_3^2$, we have that $C_2^{-1}e_2^1,C_3^{-1}e_3^1$ are the same lines, both spanned by $p_2,p_3$. Further, since $e_4^1$ is in the span of $e_4^2,e_4^3$, the line $C_4^{-1}e_4^1$ is contained in the plane spanned by $p_2,p_3,p_4$. Then it is clear that $C_4^{-1}e_4^1$ meets $C_2^{-1}e_2^1=C_3^{-1}e_3^1$. These two lines are not the same, since $p_2,p_3,p_4$ are not collinear. Then they meet in a unique point $p_1$ and this point is distinct from $p_2,p_3$, since its projection by $C_4$ is distinct from that of $p_2,p_3$.

Choose generic $X,Y\in \PP^3$ and define $x_i=C_iX,y_i=C_iY$ for $i=2,3,4$. There are unique point $x_1,y_1$ such that $x_1F^{1j}x_j=y_1F^{1j}y_j=0$ for each $j=2,3,4$. In the case of $x_1$, this is because collinear triplewise conditions say that $F^{12}x_2\sim F^{13}x_3$. To see that $F^{14}x_4$ is distinct from this line, we argue by contradiction. If $F^{12}x_2\sim F^{14}x_4$, then $F^{12}x_2\sim e_1^2\times e_1^4$ would be independent of $X$ which is not possible. Then, there is a unique camera $\widetilde{C}_1$ mapping $X$ to $x_1$, $Y$ to $y_1$, $p_1$ to $0$, $p_2$ to $e_1^2$ and $p_4$ to $e_1^4$. 

Next, we check that $\{x_1,y_1,e_1^2,e_1^4\}$ is a projective frame of $\PP^2$. If, say $x_1,e_1^2,e_1^4$, are linearly dependent, it would follow that $e_1^4F^{12}x_2=0$ or $e_1^2F^{14}x_4=0$. By triplewise constraints we would get $(e_2^1\times e_2^4)^\top x_2=0$ or $(e_4^1\times e_4^2)^\top x_4=0$, a contradiction, since $X$ is generic. If say $x_1,y_1,e_1^2$ are linearly dependent, then $F^{12}x_2\sim F^{12}y_2$ and $x_2,y_2,e_2^1$ would be linearly dependent, which is not possible, since $X$ and $Y$ were generic.  

From $\widetilde{C}_1,C_2,C_3,C_4$ we get fundamental matrices $\widetilde{F}^{ij}$, and we need to check that $\widetilde{F}^{1i}\sim F^{1i}$ for $i=2,3,4$. This is done analogously to the proof of \Cref{thm: non colin}. 
\end{proof}


\subsection{Case 4} 

\begin{theorem}[Case 3] Let $\Set{F^{ij}}$ be any set of $n\choose 2$ fundamental matrices such that the epipoles in each image are all equal. Then $\Set{F^{ij}}$ is compatible if and only if the triplewise conditions hold.  
\end{theorem}

\begin{proof} Put $P_1=\begin{bmatrix}I & 0
\end{bmatrix}$ and $P_i=\begin{bmatrix}
    [e_i^1]_\times F^{i1}&e_i^1
\end{bmatrix}$ for $i=2,\ldots,n$ as in \Cref{le: can sol}. Recall that by the collinear triplewise conditions that $F^{kj}=F^{ki}[e_i^j]_\times F^{ij}$ and in particular, $F^{1j}=-F^{1i}[e_i^1]_\times F^{ij}$. Then for $i,j\neq 1$, we have:
\begin{align}
   P_i^\top F^{ij}P_j=\begin{bmatrix}
        -F^{1i}[e_i^1]_\times \\ (e_i^1)^\top 
    \end{bmatrix}F^{ij}\begin{bmatrix}
        [e_j^1]_\times F^{j1}& e_j^1
    \end{bmatrix} =\begin{bmatrix}
        F^{1j}[e_j^1]_\times F^{j1} &  -F^{1i}[e_i^1]_\times F^{ij}e_j^1\\
        (e_i^1)^\top F^{ij} [e_j^1]_\times F^{j1} & (e_i^1)^\top F^{ij}e_j^1
    \end{bmatrix}.
\end{align}
By the fact that all epipoles in each image are the same, one can check that this matrix is skew-symmetric. 
\end{proof}

{\small
\bibliographystyle{alpha}
\bibliography{VisionBib}
}



\end{document}